\documentclass[12pt]{article}
\usepackage{epstopdf}
\usepackage{amsmath}
\usepackage{amssymb}
\usepackage{amsthm}
\usepackage{graphicx}
\usepackage{authblk}
\newcommand{\ov}{\overline}

\newtheorem{thm}{Theorem}
\newtheorem*{thm*}{Theorem}
\newtheorem{lem}{Lemma}
\newtheorem{fact}{Fact}
\newtheorem{conj}{Conjecture}

\newtheorem{defi}{Definition}


\title{Improved lower bound towards Chen-Chv\'atal conjecture}
\author[1]{Congkai Huang}
\date{August 16, 2023}
\affil[1]{School of Mathematical Sciences, Peking University, Beijing, China}
\begin{document}

\maketitle

\begin{abstract}
We prove that in every metric space where no line contains all the points, there are at least
$\Omega(n^{2/3})$ lines. This improves the previous $\Omega(\sqrt{n})$ lower bound on the 
number of lines in general metric space, and also improves the previous $\Omega(n^{4/7})$ lower bound on the number of lines
in metric spaces generated by connected graphs.
\end{abstract}

\section{Introduction}

A classic theorem in plane geometry states that every noncollinear set of $n$ points in the Euclidean space determine at least $n$ lines. This is a special case of a combinatorial theorem of De Bruijn and Erd\H{o}s~\cite{DBE} in 1948. In 2006, Chen and Chv\'{a}tal suggested that the theorem might be generalized to arbitrary metric spaces. In a metric space $(V, \rho)$,
for every pair of distinct points $a, b \in V$, the {\em line} $\ov{ab}$ is defined to be
\[
\begin{split}
\ov{ab} = \{x : & \rho(x, b) = \rho(x, a) + \rho(a, b) \text{ or } \\
& \rho(a, b) = \rho(a, x) + \rho(x, b) \text{ or } \rho(a, x) = \rho(a, b) + \rho(b, x) \}.
\end{split}
\]

If there is a line containing all the points, i.e. $\ov{ab} = V$, then $V$ is called a {\em universal line}. With this definition of the lines, Chen and Chv\'{a}tal conjectured (see \cite{CC})

\begin{conj}\label{conj.cc} In every finite metric space $(V, \rho)$, either there is a universal line, or else there are at least $|V|$ distinct lines.
\end{conj}

It was proved in \cite{ACHKS} that every finite metric space without a universal line
contains $\Omega(|V|^{1/2})$ lines. In this article we improve the lower bound to $\Omega(|V|^{2/3})$:

\begin{thm}\label{thm.main}
In every finite metric space $(V, \rho)$ without a universal line, there are at least $\Omega(|V|^{2/3})$ lines.
\end{thm}

Every connected graph $G=(V, E)$ generates a metric space $(V, \rho)$ in the natural way ---
for each pair of vertices $u$ and $v$, $\rho(u, v)$ is defined to be the length of the shortest
path from $u$ to $v$, i.e., the minimum number of edges one needs to travel from $u$ to $v$.
In \cite{ACHKS} it was also proved that every finite metric space generated by 
a connected graph either contains a universal line, or else has $\Omega(|V|^{4/7})$
lines. Our work also improves the bound in this special case to $\Omega(|V|^{2/3})$.

There are special cases of Conjecture \ref{conj.cc} where progresses are made in the
past years. Kantor \cite{K} proved that in the plane with $L_1$ metric
a non-collinear set of $n$ points induces at least $\lceil n/2 \rceil$ lines,
improving an earlier lower bound of $n/37$ by Kantor and Patk\'os \cite{KP}.
For metric spaces with a constant number of distinct distances,
Aboulker, Chen, Huzhang, Kapadia, and Supko in \cite{ACHKS} gave a $\Omega(n)$ lower bound, they also proved that every metric space on $n$ points with distances in 
$\{0, 1, 2, 3\}$ has $\Omega(n^{4/3})$ distinct lines.
Many other interesting results and stories related to the conjecture can be found
in Chv\'atal's survey \cite{C}.

In Section \ref{sect.prelim} we give some notations used throughout the paper,
and give a characterization of pairs generating the same line.
In Section \ref{sect.generators} we study the structure of the relations
for the pairs generating the same line.
When the number of lines is small, there must be a line with many different generating pairs.
The key idea allows us to improve the lower bound is a careful study
of the ``interlocked'' (we defined them as a green relation) generating pairs.
In Section \ref{sect.green_cpt} we study the structure of each component
connected by the green relations. This study allow us to find many lines
if the green component is large.
Finally in Section \ref{sect.proof_main} we combine all the pieces together
and prove Theorem \ref{thm.main}.

\section{Notations and preliminaries}\label{sect.prelim}

For distinct points $a_0$, $a_1$, ..., $a_k$,
\[ [a_0 a_1 ... a_k] \;\;\mbox{means}\;\; \rho(a_0, a_k) = \sum_{i=0}^{k-1} \rho(a_i, a_{i+1}) \]
With this notation, for a metric space, we call three distinct points $a$, $b$, $c$ {\em collinear} if $[acb]$ or $[cab]$ or $[abc]$, and the line $\ov{ab}$ is the set of points consisting of $a$, $b$, and any $c$
that is collinear with $a$ and $b$.

\begin{defi}\label{defi.collinear_set_triple}
When $k \ge 1$ and $[a_0 a_1 ... a_k]$ happens, we call $(a_0, a_1, \dots, a_k)$
a {\em collinear sequence}, and call the set $\{a_0, a_1, \dots, a_k \}$ a collinear set.

For a collinear triple $\{a, b, c\}$, when $[acb]$ we say $c$ is {\em between} $a$ and $b$,
when $[cab]$ we say $c$ is on the {\em $a$-side} of $\{a, b\}$,
and when $[abc]$ we say $c$ is on the {\em $b$-side} of $\{a, b\}$.
In the latter two cases we say $c$ is {\em outside} $\{a, b\}$.
\end{defi}

We also use the standard notations about sequences.

\begin{defi}
For a sequence $\pi = (x_1, x_2, \dots, x_k)$,
its {\em reverse} is $\pi^R = (x_k, x_{k-1}, \dots, x_1)$.

For sequences $\pi = (x_1, x_2, \dots, x_k)$ and $\sigma = (y_1, y_2, \dots, y_s)$,
their {\em concatenation} is $\pi \circ \sigma = (x_1, x_2, \dots, x_k, y_1, y_2, \dots, y_s)$.
\end{defi}

The following facts are obvious. We list them and will use them frequently,
sometimes implicitly.

\begin{fact}\label{fact.collinear}
 If $(V,\rho)$ is a metric space and $a,b,c,d,a_i$ for $i=0,...,k$ and $b_j$ for $j=1,...,s$ are distinct points of $V$, then

(a) $[abc] \Leftrightarrow [cba]$;

(b) $[abc]$ and $[acb]$ cannot both hold;

(c) $[abc]$ and $[acd]$ implies $[abcd]$;

(d) more generally, $[a_0\dots a_k]$ and $[a_i b_1 \dots b_s a_{i+1}]$
imply $[a_0 \dots a_i b_1 \dots b_s a_{i+1} \dots a_k]$;

(d) $[a_0a_1...a_k]$ implies $\rho(a_s, a_t) = \sum_{i=s}^{t-1} \rho(a_i, a_{i+1})$ 
for every pair $s$ and $t$ such that $s < t$;

(e) $[a_0a_1...a_k]$ implies $[a_i a_j a_s]$ for every pair $i$ and $j$ such that $0 \leq i < j < s \leq k$.
\end{fact}

The following is more general than Fact \ref{fact.collinear} (a) and (b).

\begin{fact}\label{fact.two_ordering}
The elements of every collinear set can form exactly two collinear sequences
and they reverse each other.
\end{fact}

\begin{proof}
Suppose the elements of $W$ are $\{a_0, \dots, a_k\}$
and $[a_0 a_1 ... a_k]$,
then the pair $\{a_0, a_k\}$ is the unique pair in $W$ that achieves the diameter of $W$.
So, every collinear ordering either starts with $a_0$ and ends with $a_k$,
or starts with $a_k$ and ends with $a_0$. In either case, 
the ordering of positions of all the other points are determined by their distance to $a_0$.
\end{proof}

\begin{fact}\label{fact.star}
If $u, v, w, s$ are four distinct points of $V$ satisfying
$[usv]$, $[vsw]$, and $[wsu]$, then $\{u, v, w\}$ is not collinear.
\end{fact}

\begin{proof}
Let $x = \rho(u, s)$, $y = \rho(v, s)$, $z = \rho(w, s)$.
Then $\rho(u, v) = x+y$, $\rho(v, w) = y+z$, and $\rho(w, u) = z+x$.
The sum of any two of them is greater than the third. 
\end{proof}

\begin{defi}
Let $(V, \rho)$ be a metric space and let $L$ be a line of $(V, \rho)$,
define the set of its {\em generating pairs}
\begin{equation}\label{eq.def_generator_pairs}
K = K(L) := \left\{ \{a, b\} \in \binom{V}{2} : \ov{ab} = L \right\}.
\end{equation}
\end{defi}

For the rest of this work, we use $ab$ to denote the binary set $\{a, b\}$.

First we discuss the possible relations between two pairs generating the same line $L$.
\footnote{Some of the classifications here are similar to those in Section 6 of \cite{ACHKS}.}
In order to formally do this, we introduce a little notation.

\begin{defi}
for a pair of points $e = \{a, b\}$, we denote $\rho(e) = \rho(a, b)$.
\end{defi}

\begin{fact}\label{fact.relations}
For any $e_1, e_2 \in K(L)$ where $\rho(e_1) \ge \rho(e_2)$, exactly one of the following happens.

(o) {\em ordered relation}:

(o.1) $e_1 = e_2$;

(o.2) $|e_1 \cap e_2| = 1$, they can be written as $e_1 = \{a, b\}$, $e_2 = \{a, c\}$, and $[acb]$;

(o.3) $|e_1 \cap e_2| = 0$, they can be written as $e_1 = \{a, b\}$, $e_2 = \{c, d\}$,
and $[acdb]$.

(b) {\em blue relation}:

(b.1) $|e_1 \cap e_2| = 1$, they can be written as $e_1 = \{a, b\}$, $e_2 = \{a, c\}$,  and $[bac]$;

(b.2) $|e_1 \cap e_2| = 0$, they can be written as $e_1 = \{a, b\}$, $e_2 = \{c, d\}$,
and $[abcd]$.

(g) {\em green relation}: $|e_1 \cap e_2| = 0$, they can be written as $e_1 = \{a, b\}$, $e_2 = \{c, d\}$,
and $[acbd]$.

(r) {\em red relation}: 
$|e_1 \cap e_2| = 0$, they can be written as $e_1 = \{a, b\}$, $e_2 = \{c, d\}$, 
and there are positive reals $x$ and $y$
such that 
$\rho(a, b) = \rho(c, d)=x$,
$\rho(a, c) = \rho(b, d) = y$, and $\rho(a, d) = \rho(b, c) = x+y$.

(p) {\em purple relation}:
$|e_1 \cap e_2| = 0$, they can be written as $e_1 = \{a, b\}$, $e_2 = \{c, d\}$, 
and there are positive reals $x$ and $y$
such that 
$\rho(a, c) = \rho(b, d)=x$,
$\rho(a, d) = \rho(b, c)=y$, and $\rho(a, b) = \rho(c, d) = x+y$.
\end{fact}

These relations are depicted in Figure \ref{fig.relations}.
(It is helpful to see the corresponding cases in
Figure \ref{fig.relations} when reading the following proof;
it is also helpful to use the figure when reading the rest of the paper.)
 
\begin{figure}[h!]
\begin{center}
\includegraphics[width=5in]{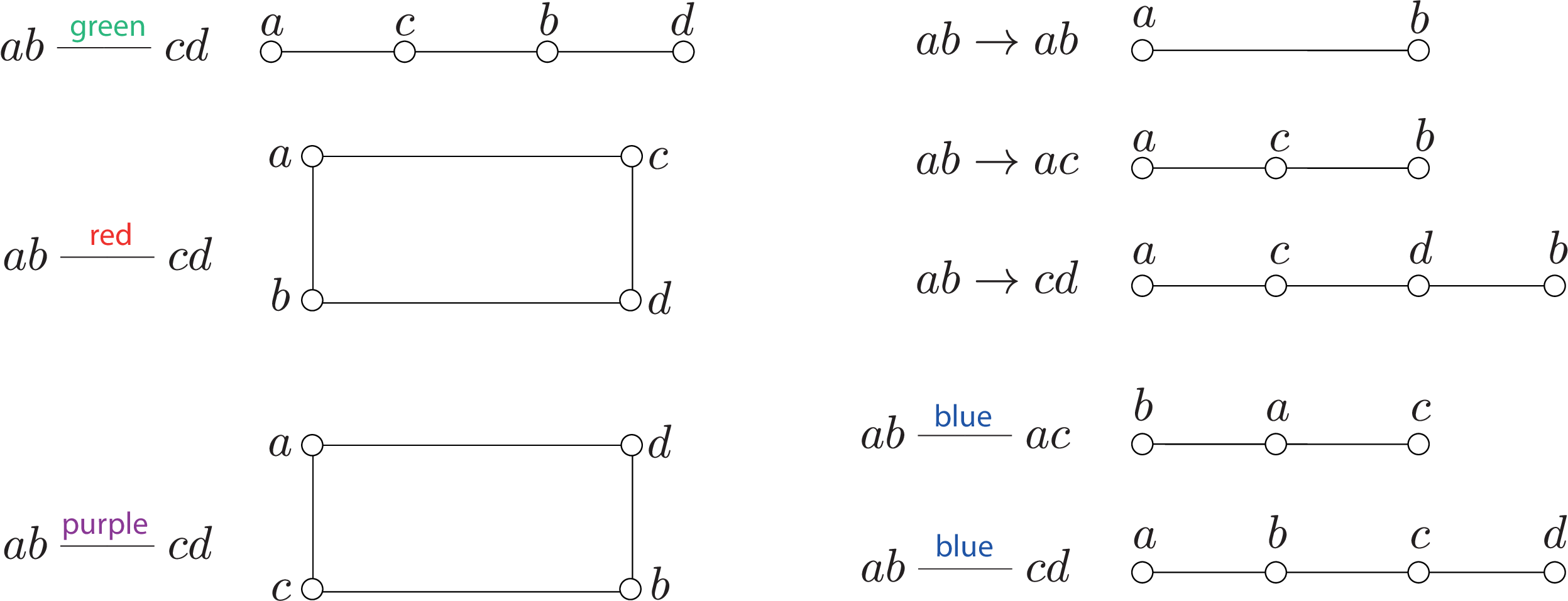}
\caption{relations for $e_1, e_2 \in K(L)$.}
\label{fig.relations}
\end{center}
\end{figure}

\begin{proof}
We only need to show $e_1$ and $e_2$ has one of the relations.
The uniqueness obviously follows Fact \ref{fact.two_ordering}.

When $|e_1\cap e_2| = 2$, obviously (o.1) happens.
When $|e_1 \cap e_2| = 1$,
write $e_1 = \{a, b\}$ and $e_2 = \{a, c\}$. 
Since $c \in \ov{ac} = \ov{ab} = L$,
the set $\{a, b, c\}$ is collinear, one of $[bac]$, $[acb]$, and $[abc]$ happens.
The first situation is (b.1), the second is (o.2), and the last does not happen since we 
assumed $\rho(e_1) \ge \rho(e_2)$. 

Finally, $|e_1\cap e_2| = 0$. Write $e_1 = \{a, b\}$, $e_2 = \{c, d\}$.
Since $\{c, d\} \subseteq \ov{cd} = \ov{ab} = L$,
both $\{a, b, c\}$ and $\{a, b, d\}$ are collinear. 
Similarly, $\{a, c, d\}$ and $\{b, c, d\}$ are collinear.
Since $\rho(a, b) \ge \rho(c, d)$, we may choose the letters so that
either $\rho(a, b)$ is one of the largest distances among the fours points,
or $\rho(a, d)$ is one of  the largest. We discuss two cases.

{\em Case 1.} $\rho(a, b)$ is one of the largest distances among $\{a, b, c, d\}$.
Since $\{a, b, c\}$ and $\{a, b, d\}$ are both collinear,
both $c$ and $d$ are between $a$ and $b$, i.e., $[acb]$ and $[adb]$.
We may choose the letters so that $\rho(a, c) \le \rho(a, d)$.
Let $\rho(a, c)=x$, $\rho(a, d)=y$, $0<x\le y$.
Because $a \in \ov{ab} = \ov{cd} = L$, we have $[adc]$, $[acd]$, or $[cad]$.
$[adc]$ contradicts the assumption $\rho(a, c) \le \rho(a, d)$.

{\em Case 1.1.} $[acd]$. Fact \ref{fact.collinear} (c) together with $[adb]$ imply $[acdb]$;
and this is (o.3).

{\em Case 1.2.} $[cad]$. We have $\rho(c, d) = x+y$.
Let $x' = \rho(b, d)$ and $y' = \rho(b, c)$.
By the assumption $[acb]$ and $[adb]$, we have
\begin{equation}\label{eq.generators_01}
x+y' = \rho(a, c)+\rho(c, b) = \rho(a, b) = \rho(a, d) + \rho(b, d) = y+x'.
\end{equation}
$x \le y$ implies $x' \le y'$. $\{b, c, d\}$ is collinear,
so $[dcb]$, $[cdb]$, or $[cbd]$. $[dcb]$ contradicts the fact $x' \le y'$.
$[cdb]$, together with $[acb]$ and Fact \ref{fact.collinear} (c) implies $[acdb]$;
this is (o.3). Finally, $[cbd]$ implies $x+y = \rho(c, d) = x'+y'$,
together with \eqref{eq.generators_01}, we have $x=x'$ and $y=y'$;
this is situation (p).

{\em Case 2.} $\rho(a, d)$ is one of the largest distances among $\{a, b, c, d\}$.
Since every triple in $\{a, b, c, d\}$ is a collinear set, we have $[abd]$ and $[acd]$.
Now, we discuss on the collinear set $\{a, b, c\}$.

{\em Case 2.1.} $[abc]$. Together with $[acd]$ and Fact \ref{fact.collinear} (c)
we have $[abcd]$; this is situation (b.2).

{\em Case 2.2.} $[acb]$. Together with $[abd]$ and Fact \ref{fact.collinear} (c)
we have $[acbd]$; this is situation (g).

{\em Case 2.3.} $[bac]$. Let $x = \rho(a,b)$ and $y = \rho(a, c)$, 
then $\rho(b, c) = x+y$. Let $x' = \rho(c, d)$ and $y' = \rho(b, d)$.
$[abd]$ and $[acd]$ imply
\begin{equation}\label{eq.generators_02}
x+y' = \rho(a, b) + \rho(b, d) = \rho(a, d) = \rho(a, c) + \rho(c, d) = x'+y.
\end{equation}
Note that $\{b, c, d\}$ is collinear. If $[bcd]$, together with $[bac]$
and Fact \ref{fact.collinear} (c) we have $[bacd]$ contradicts the
maximality assumption of Case 2. The same happens when $[cbd]$.
So we have $[cdb]$, so
\[
\rho(b, c) = x+y = x'+y'.
\]
Together with \eqref{eq.generators_02} imply
$x=x'$ and $y=y'$. It is easy to see this is situation (r).
\end{proof}

Next, we define a graph on the pairs of points.
Throughout the paper we use {\em points} for the elements
of the metric space, and {\em vertices} for the vertices of
the graph, so each vertex is a pair of points.

\begin{defi}
Define a relation $\mathcal{P}_L = (K(L), \preccurlyeq)$
as $e_2 \preccurlyeq e_1$ if and only if
$e_1$ and $e_2$ satisfy one of the ordered relations
(o) in Fact \ref{fact.relations}.

Define a coloured graph $\mathcal{G}_L$ on $K(L)$,
$e_1$ and $e_2$ has an edge with colour
blue (respectively, green, red, purple)
whenever they have the blue (respectively, green, red, purple)
relation as in Fact \ref{fact.relations},
and we denote this by $e_1 \sim_b e_2$
(respectively, $e_1 \sim_g e_2$, $e_1 \sim_r e_2$, $e_1 \sim_p e_2$).
\end{defi}

By Fact \ref{fact.collinear}, it is easy to check that
$\mathcal{P}_L$ is a partially ordered set (poset).
Now we give a partition of $K(L)$.

\begin{defi}
For $e \in K(L)$, define $\ell(e)$
be the length of the longest chain in the poset $\mathcal{P}_L$
with $e$ as its maximum element.

For a positive integer $k$, define 
\[
\mathcal{P}_L^{(k)} = \ell^{-1}(k) = \{e \in K(L): \ell(e) = k\}.
\]
and $\mathcal{G}_L^{(k)}$ be the (coloured) induced subgraph
of $\mathcal{G}_L$ on $\mathcal{P}_L^{(k)}$.

Let $h(L)$ be the largest integer $k$
for which $\mathcal{P}_L^{(k)} \neq \emptyset$.
This is the {\em height} of $\mathcal{P}_L$.
\end{defi}

It is a well known fact in partially ordered sets that

\begin{fact}\label{fact.antichain}
For each positive integer $k$, $\mathcal{P}_L^{(k)}$ is an antichain.
\end{fact}

\begin{proof}
Suppose $e_1 \neq e_2 \in \mathcal{P}_L^{(k)}$
and $e_1 \preccurlyeq e_2$.
By definition $\ell(e_1) = \ell(e_2) = k$.
There is a chain $C$ with maximum element $e_1$
and $|C|=k$.
But then $C \cup \{e_2\}$
is a chain with maximum element $e_2$
and size $k+1$, contradicts the fact $\ell(e_2) = k$.
\end{proof}

Consequently, since Fact \ref{fact.relations} tells us
that any two elements in $K(L)$ either are in the ordered relation or they form one of the coloured edges, we have

\begin{fact}\label{fact.complete_on_level}
For every integer $k$ with $1 \le k \le h(L)$,
the graph $\mathcal{G}_L^{(k)}$ is a complete graph (with coloured edges).
\end{fact}

We also note that it is clear from Fact \ref{fact.relations} and Fact \ref{fact.collinear} that

\begin{fact}\label{fact.level_k_inner}
For every $ab \in \mathcal{P}_L^{(k)}$, there are at least $k-1$ inner points collinear with $a$ and $b$,
i.e., there are distinct points $x_1, x_2, \dots, x_{k-1}$ in $V \setminus \{a, b\}$ such that
\[
[a x_1 x_2 \dots x_{k-1} b].
\]
\end{fact}

\section {The structure of $K(L)$}\label{sect.generators}

\begin{defi}
We call an element $e \in K(L)$ purple if 
$e \sim_p f$ for some $f \in K(L)$,
and denote $U(L)$ the set of all the purple elements;
call an element $e \in K(L)$ red if
$e \sim_r f$ for some $f \in K(L)$,
and denote $D(L)$ the set of all the red elements.
\end{defi}

\begin{fact}\label{fact.red_element}
For a red element $e = \{a, b\} \in D(L)$,
no point (other than $a$ and $b$) in $V$ is between $a$ and $b$;
furthermore, $e$ is a minimal element in the poset
$\mathcal{P}_L$.
\end{fact}

\begin{proof}
Indeed, suppose $[avb]$.
Let $x_1 = \rho(a, v)$, $x_2 = \rho(v, b)$,
and $x=x_1+x_2 = \rho(a, b)$.
Pick an element $f \in K(L)$
such that $e \sim_r f$.
By Fact \ref{fact.relations}, $[cad]$ and $[acd]$,
together with $[avb]$ we have $[vac]$ and $[vbd]$.
Also by Fact \ref{fact.relations},
$\rho(a, c) = \rho(b, d) = y$, $\rho(c, d) = x$.
We have
\[
\begin{split}
[vac] \Rightarrow & \rho(v, c) = x_1 + y, \\
[vbd] \Rightarrow & \rho(v, b) = x_2 + y, \\
& \rho(c, d) = x.
\end{split}
\]
It is easy to see that the sum of any two of the
distances among $\{v, c, d\}$ is larger than the third.
So $v \not\in \ov{cd} = L$.
But $[avb]$ implies $v \in \ov{ab} = L$, a contradiction.

It immediately follows that $e$ is a minimal element
in $\mathcal{P}_L$ --- were there $f \in K(L)$, $f \neq e$ and $f \preccurlyeq e$,
by Fact \ref{fact.relations} (o) there would be another point in $L$ between $a$ and $b$.
\end{proof}

\begin{fact}\label{fact.purple_basic}
For a purple element $e = \{a, b\} \in U(L)$,
every point in $L$ (other than $a$ and $b$) is between $a$ and $b$,
i.e., for every $v \in L$, we have $[avb]$.
\end{fact}

\begin{proof}
Since $e$ is purple, there is $f = \{c, d\} \in K(L)$
such that $e \sim_p f$.
Suppose $v \in L = \ov{ab}$ and $v$ is not between $a$ and $b$.
Without loss of generality, $[abv]$.
By Fact \ref{fact.relations} (p),
we have $[acb]$ and $[adb]$;
then by Fact \ref{fact.collinear} we have $[cbv]$ and $[dbv]$;
also Fact \ref{fact.relations} (p) tells us $[cbd]$,
so, by Fact \ref{fact.star}, $\{v, c, d\}$ is not collinear,
contradicts the fact $v \in L = \ov{cd}$.
\end{proof}

\begin{fact}\label{fact.purple_element}
When the set of purple elements $U(L) \neq \emptyset$, we have

(a) $U(L)$ is the set of all maximal elements in the poset $\mathcal{P}_L$;

(b) For any $e \neq f \in U(L)$, $e \sim_p f$;

(c) For any $e \in U(L)$ and $f \in K(L) \setminus U(L)$, $f \preccurlyeq e$.

(d) $U(L) = \mathcal{P}_L^{(h(L))}$, the last level of the poset $\mathcal{P}_L$.
\end{fact}

\begin{proof}
We first note that,
\begin{equation}\label{eq.purple_max_01}
\forall e \in U(L) \forall f \in K(L), f \preccurlyeq e \;\;\text{or}\;\; e \sim_p f.
\end{equation}
This can be seen by observing Fact \ref{fact.relations}.
The pair $\{e, f\}$ is in one of the relations, but anything other than
the two listed in \eqref{eq.purple_max_01}
leads to a contradiction with Fact \ref{fact.purple_basic}.

Denote $M$ the set of maximal elements in $\mathcal{P}_L$.

(a).  For every $e \in U(L)$, by \eqref{eq.purple_max_01},
there cannot exist $f \neq e$, $f \in K(L)$, such that
$e \preccurlyeq f$, so $e \in M$. Thus we have $U(L) \subseteq M$.
On the other hand, 
since $U(L) \neq \emptyset$, pick some $e^* \in U(L)$.
Suppose $f$ is any other maximal element so $f \preccurlyeq e$ does not hold.
By \eqref{eq.purple_max_01}, $e \sim_p f$, thus $f$ is purple by the definition.
Therefore, $M \subseteq U(L)$. 

(b). Now we know that both $e$ and $f$ are maximal elements, so they are not comparable,
and \eqref{eq.purple_max_01} implies $e \sim_p f$.

(c). Since $f$ is not purple, $e \not\sim_p f$, so \eqref{eq.purple_max_01} implies $f \preccurlyeq e$.

(d). If $U(L) = K(L)$, no two elements are comparable by (a), and we have $\ell(e) = 1$ for all $e \in K(L)$;
so $h(L) = 1$ and $\mathcal{P}_L^{(1)} = K(L) = U(L)$.
Otherwise, consider a longest chain without a purple element
\[
C = (e_1, e_2, \dots, e_t).
\]
Given any non-purple element $f \in K(L)$,
(c) implies that the longest chain with $f$ as its maximum element
does not contain purple elements, so we have
\begin{equation}\label{eq.purple_max_02}
\forall f \in K(L) \setminus U(L),\;\; \ell(f) \le t.
\end{equation}
On the other hand, for every $e \in U(L)$, 
(a) implies the longest chain with $f$ as its maximum element
does not contain any other purple elements,
so $\ell(e) \le t+1$; but $C \cup \{e\}$ is a chain, so
\begin{equation}\label{eq.purple_max_03}
\forall e \in K(L), \;\; \ell(e) = t+1.
\end{equation}
\eqref{eq.purple_max_02} and \eqref{eq.purple_max_03} imply
$h(L) = t+1$ and $U(L) = \mathcal{P}_L^{(h(L))}$.
\end{proof}

Now we turn to the study $\mathcal{P}_L^{(k)}$ for $2 \le k \le h(L)-1$.
By Facts \ref{fact.red_element} and \ref{fact.purple_element},
there are no red nor purple elements in such levels;
Fact \ref{fact.complete_on_level} implies that $\mathcal{G}_L^{(k)}$
is a complete graph with blue and green edges.

\begin{defi}
For a line $L$ and an index $k$ with $2 \le k \le h(L)-1$,
$\mathcal{R}_L^{(k)}$ is the green sub-graph of $\mathcal{G}_L^{(k)}$,
it has the vertex set $\mathcal{P}_L^{(k)}$ and has all the green edges of $\mathcal{G}_L^{(k)}$.
Denote $Q_L^{(k)}$ the set of isolated vertices in $\mathcal{R}_L^{(k)}$,
denote $c_L(k)$ the number of connected components of size at least 2 in $\mathcal{R}_L^{(k)}$,
we call them the {\em green components in level $k$},
denote $P_L^{(k, i)}$ ($i = 1, 2, \dots, c_L(k)$) the vertex set for each green component.

For every subset $U \subseteq \mathcal{P}_L^{(k)}$,
denote $V(U) \subseteq V$ the union of elements (each is a pair of points in $V$) of $U$,
i.e., all the endpoints of generating pairs of $L$ in $U$.
\end{defi}

\begin{fact}\label{fact.Q_size}
For every $2 \le  k \le h(L)-1$, $\left|Q_L^{(k)}\right| \le |V|/(k-1)$.
\end{fact}

\begin{proof}
It suffices to show that there are $|Q_L^{(k)}|$ pairwise disjoint sets, each with
 $k-1$ points, and altogether they have no more than $|V|$ points. Fact \ref{fact.level_k_inner} implies that each element in $Q_L^{(k)}$ has
$k-1$ ``inner'' points, so we just need to show that these inner points are all distinct.
Suppose $ab, cd \in Q_L^{(k)}$ and $[axb]$, $[cyd]$, we want to show that $x \neq y$.
Indeed, $ab$ and $cd$ are isolated vertices in $\mathcal{R}_L^{(k)}$,
so $ab \sim_b cd$.
By the categorization in Fact \ref{fact.relations},
We have either $[axbcyd]$, or, in the special case of $d=a$, $[bxayc]$.
It is clear in both cases $x \neq y$.
\end{proof}

\section{The structure of a green component}\label{sect.green_cpt}

In this section, we fix a line $L$ and an index $k$ with $2 \le  k \le h(L)-1$,
denote $\mathcal{U} = \mathcal{P}_L^{(k)}$,
and denote the green subgraph $\mathcal{R}_L^{(k)}$ by $\mathcal{R}$.

\begin{defi}
For a subset $\mathcal{W} \subseteq \mathcal{U}$,
we call a permutation $\pi$ of the set of its endpoints $V(\mathcal{W})$
a {\em collinear ordering} for $\mathcal{W}$ if $\pi$ is a collinear sequence.
For each pair $ab \in \mathcal{W}$,
where $a$ comes before $b$ in $\pi$,
we call $a$ the {\em opening point} of $ab$,
and $b$ the corresponding {\em closing point}.
Since any two pairs $ab, cd \in \mathcal{W}$
has either $ab \sim_b cd$ or $ab \sim_g cd$,
all the opening points are distinct;
we call the sequence of opening points,
sorted by their position in $\pi$ from the earliest to the latest,
the {\em opening sequence} of $\pi$.

Let $a$ be an opening point and $b$ be its corresponding closing point,
and $v$ be a point in $L = \ov{ab}$.
We say $v$ is on the {\em left side} of $a$ in $\pi$
if $[vab]$, otherwise (when $[avb]$ or $[abv]$) $v$ is on the {\em right side} of $a$.
We say $v$ is on the {\em left side} of $b$ in $\pi$
if $[vab]$ or $[avb]$, otherwise (when $[abv]$) $v$ is on the {\em right side} of $b$.
\end{defi}

We will prove the existence of a collinear ordering for every connected
subgraph of $\mathcal{R}$.
Before this, we first discuss some properties of such an ordering if
one exists, as we will need them in the inductive proof.

Lemma \ref{lem.insertion_order} analyzes the opening and closing points
on a collinear ordering, and the relation of a single point $v \in L$
to the opening-closing pairs.
In turn, Lemma \ref{lem.insertion} gives the shape of the line $L$
with respect to a collinear ordering.
(See Figure \ref{fig.green_cpt}.)

\begin{figure}[h!]
\begin{center}
\includegraphics[width=3.2in]{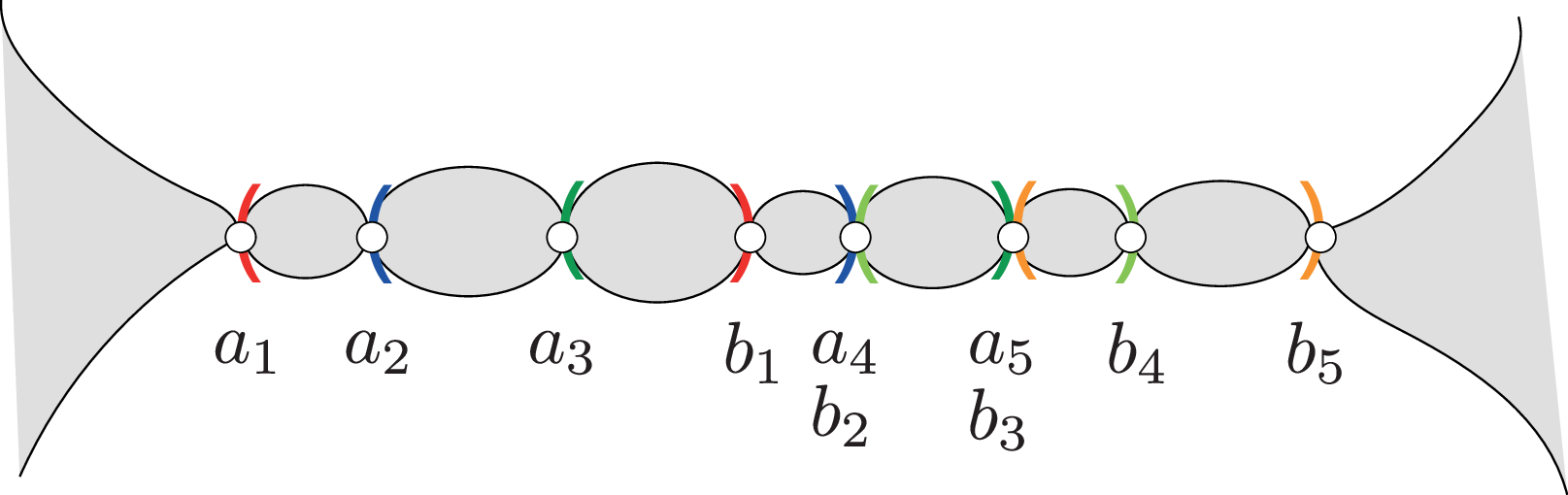}
\caption{The shaded area contains all points of $L$. 
In a collinear ordering, $a_1, a_2, \dots, a_5$ are the opening points,
$b_1, b_2, \dots, b_5$ are corresponding closing points.
$V = \ov{a_ib_i}$ for all $i=1, 2, 3, 4, 5$.}
\label{fig.green_cpt}
\end{center}
\end{figure}

\begin{lem}\label{lem.insertion_order}
Let $H$ be a connected subgraph of $\mathcal{R}$ with order at least 2,
$\pi$ a collinear ordering for the vertex set of $H$,
with $(a_1, a_2, \dots, a_t)$ as its opening sequence;
let $b_i$ be the corresponding closing point to $a_i$ $i = 1, 2, \dots, t$.
Then

(a) The $b_i$'s are distinct and their order in $\pi$, from left to the right, is $b_1, b_2, \dots, b_t$;

(b) $a_i b_i \sim_g a_{i+1}b_{i+1}$ for every $1 \le i < t$ and their order on $\pi$ is $a_i, a_{i+1}, b_i, b_{i+1}$,
and these four points are distinct;

(c) For every point $v \in L$, either $v$ appears on $\pi$, 
or $\pi$ has a partition $\pi = \sigma \circ \tau$ such that $v$
is on the right side of all points in $\sigma$ and on the left side of all
points in $\tau$. (Here $\sigma$ and $\tau$ can be the empty sequence.)
\end{lem}

\begin{proof} 
(a). We only need to show $b_i$ comes strictly before $b_{i+1}$
for $1 \le i < t$. Assume $b_{i+1}$ comes before $b_i$ or $b_{i+1} = b_i$,
then $a_ib_i$ and $a_{i+1}b_{i+1}$ have one of the ordered relations in
Fact \ref{fact.relations}, contradicts the fact that they have blue or green relation.

(b). By (a) and the definition of opening and closing points, 
among the four points, 
$a_i$ is the first and $b_{i+1}$ is the last.
Now if $b_i$ is before $a_{i+1}$ or $b_i=a_{i+1}$,
$a_j \ge b_k$ for all $j>i$ and $k \le i$,
so all the edges between $\{a_kb_k: k \le i\}$ and $\{a_jb_j: j > i\}$
are blue, which contradicts the fact that $H$ is connected by green edges.

(c). Suppose $v \in L$ but does not appear on $\pi$. 
We only need to show there cannot be two consecutive points $x$ and $y$ in $\pi$
such that $x$ is to the left of $y$, $v$ is on the left side of $x$ but on the right side of $y$.
Suppose the contrary and assume $x$ and $y$ is the earliest such a pair. We have the following cases. (Note that since a closing point can be the opening point for another pair, the cases might
overlap.)

\noindent {\em Case 1.} Both $x$ and $y$ are opening points. 
We have $x = a_i$ and $y = a_{i+1}$ for some $1 \le i < t$.
By (b) we have 
\begin{equation}\label{eq.insertion_order_01}
[a_i a_{i+1} b_i b_{i+1}]
\end{equation}
By definition, $v$ is on the left side of $x = a_i$, so we have
\begin{equation}\label{eq.insertion_order_02}
[v a_i b_i]
\end{equation}
\eqref{eq.insertion_order_01} and \eqref{eq.insertion_order_02}
imply
\begin{equation}\label{eq.insertion_order_03}
[v a_i a_{i+1} b_i].
\end{equation}
The point $v$ is on the right side of $y = a_{i+1}$,
so $[a_{i+1}vb_{i+1}]$ or $[a_{i+1}b_{i+1}v]$.
However, $[a_{i+1}vb_{i+1}]$ and $[a_{i}a_{i+1}b_{i+1}]$ in \eqref{eq.insertion_order_01}
imply $[a_{i}a_{i+1} v b_{i+1}]$, thus $[a_{i}a_{i+1} v]$, contradicts
$[va_ia_{i+1}]$ in \eqref{eq.insertion_order_03};
and $[a_{i+1}b_{i+1}v]$ and $[a_{i+1} b_i b_{i+1}]$
in \eqref{eq.insertion_order_01} imply $[a_{i+1}b_ib_{i+1}v]$,
thus $[a_{i+1}b_iv]$, contradicts $[va_{i+1}b_i]$ in
\eqref{eq.insertion_order_03}.

\noindent {\em Case 2.} Both $x$ and $y$ are closing points.
This is completely similar to the previous case.
Or, we can reverse $\pi$ and reduce this to the previous case.

\noindent {\em Case 3.} $x=a_j$ is an opening point and $y=b_i$ is a closing point.
 Since $a_j$ and $b_i$ are consecutive in $\pi$, by (a) and (b), we have $i<j$ and
\begin{equation}\label{eq.insertion_order_04}
[a_i a_{j} b_i b_{j}]
\end{equation}
on $\pi$.
$v$ is on the left side of $x=a_j$ and on the right side of $y=b_i$,
by definition, we have
\begin{equation}\label{eq.insertion_order_05}
[v a_j b_j] \;\; \mbox{and} \;\; [a_ib_iv].
\end{equation}
The first item in \eqref{eq.insertion_order_05} 
and $[a_jb_ib_j]$ in \eqref{eq.insertion_order_04}
imply $[v a_j b_i b_j]$, thus
\begin{equation}\label{eq.insertion_order_06}
[va_jb_i].
\end{equation}
The second item in \eqref{eq.insertion_order_05} 
and $[a_ia_jb_i]$ in \eqref{eq.insertion_order_04}
imply $[a_i a_j b_i v]$, thus $[a_jb_iv]$,
contradicts \eqref{eq.insertion_order_06}.

\noindent {\em Case 4.} $x=b_i$ is a closing point and $y=a_j$ is an opening point.
 Since $x$ and $y$ are consecutive in $\pi$, and by (a) and (b),
 it is easy to see $i+1<j$, and we have
\begin{equation}\label{eq.insertion_order_07}
[a_i a_{i+1} b_i a_j b_{i+1} b_j].
\end{equation}
Now $v$ is on the left side of $x=b_i$ and on the right side of $y = a_j$,
we have
\begin{equation}\label{eq.insertion_order_08}
\begin{split}
[va_ib_i] \;\; & \mbox{or} \;\; [a_ivb_i] \\
& \mbox {and} \\
[a_jb_jv] \;\; & \mbox{or} \;\; [a_jvb_j].
\end{split}
\end{equation}
$[a_i v b_i]$ and $[a_i b_i a_j b_j]$ in \eqref{eq.insertion_order_07}
imply $[a_i v b_i a_j b_j]$, 
thus $[va_jb_j]$, contradicts the fact that $v$ is on the right side of $a_j$.

$[a_j v b_j]$ and $[a_i b_i a_j b_j]$ in \eqref{eq.insertion_order_07}
imply $[a_i b_i a_j v b_j]$, 
thus $[a_ib_iv]$, contradicts the fact that $v$ is on the left side of $b_i$.

So, \eqref{eq.insertion_order_08} becomes
\begin{equation}\label{eq.insertion_order_09}
[va_ib_i] \;\; \mbox{and} \;\; [a_jb_jv]
\end{equation}
The first item above means $v$ is on the left side of $a_i$,
and by the minimality of $x$, $v$ is on the left side of all points
in $\pi$ between $a_i$ and $b_i$,
in particular, $v$ is on the left side of $a_{i+1}$,
i.e.,
$[va_{i+1}b_{i+1}]$.
This, together with $[a_{i+1} a_j b_{i+1}]$ in \eqref{eq.insertion_order_07}, 
implies $[va_{i+1} a_j b_{i+1}]$, thus
\begin{equation}\label{eq.insertion_order_10}
[va_jb_{i+1}].
\end{equation}
The second item in \eqref{eq.insertion_order_09} and
$[a_jb_{i+1}b_j]$ in \eqref{eq.insertion_order_07}
imply $[a_j b_{i+1} b_j v]$, thus $[a_j b_{i+1} v]$, contradicts
\eqref{eq.insertion_order_10}.
\end{proof}

The following lemma is well-known in graph theory.
\footnote{
This is pointed to us by Va\v{s}ek Chv\'atal --- If the graph $K_0$ with no vertices is considered connected, 
then the unique vertex of $K_1$ is considered not to be a cut point,
and the lower bound on the order of $G$ may be dropped in the lemma.
Arguments for and against declaring $K_0$ connected are presented on pages 42-43 of
\cite{HR}.
}

\begin{lem}\label{lem.non_cut_point}
Every connected graph $G$ with order at least 2 has a vertex that is not a cut vertex,
i.e., $G - v$ is still connected.
\end{lem}

\begin{proof}
Consider a spanning tree $T$ of $G$.
Being a tree with at least two vertices,
$T$ has at least two leaves.
Any leaf $v$ of $T$ has the property that $T-v$ is connected,
so $G-v$ is connected.
\end{proof}

\begin{lem}\label{lem.insertion}
For every connected subgraph $H$ of $\mathcal{R}$,
a collinear ordering $\pi$ for the vertex set of $H$,
and every point $v \in L$ that does not appear on $\pi$,
$v$ can be inserted into $\pi$ in a unique way to form a collinear sequence.
\end{lem}

\begin{proof}
The uniqueness of the insertion position follows from Fact \ref{fact.two_ordering}.
We show the existence of the position by induction on the order of $H$. The base case when
$H$ has one vertex $e = ab$ is clear --- 
since $v \in L = \ov{ab}$, we have $[vab]$, $[avb]$, or $[abv]$.
Now assume $H$ has order at least two and the statement holds
for every connected subgraph of $H$ with a smaller order.
Let $(a_1, a_2, \dots, a_t)$ be the opening sequence for $\pi$,
and $b_i$ be the corresponding closing point to $a_i$.
Let $\sigma$ and $\tau$ be the partition of $\pi$
assured by Lemma \ref{lem.insertion_order} (c).

\noindent {\em Case 1.} $\sigma$ is the empty sequence.

So, $\tau = \pi$ and $v$ is on the left side of all the points
(especially all the opening points) in $\pi$.
Thus, 
\begin{equation}\label{eq.insertion_left_00}
[va_ib_i], \;\; \forall 1 \le i \le t.
\end{equation}
By Lemma \ref{lem.insertion_order} (b), $H - a_tb_t$ is connected.
The sub-sequence $\pi'$ of $\pi$ formed by $\{a_1, b_1, a_2, b_2, \dots, a_{t-1}, b_{t-1}\}$
is a collinear ordering for $H - a_tb_t$.
Thus, the inductive hypothesis assures that $v$ can be inserted into $\pi'$ to form
a collinear sequence $\pi''$. Since $v$ is on the $a_1$-side of $\{a_1, b_1\}$,
we conclude that
\begin{equation}\label{eq.insertion_left_01}
\mbox{$\pi'' = (v) \circ \pi'$ is a collinear sequence ends with $b_{t-1}$.}
\end{equation}
By Lemma \ref{lem.insertion_order} (b), $[a_t b_{t-1} b_t]$,
together with $[va_tb_t]$ in \eqref{eq.insertion_left_00}, we have $[va_tb_{t-1}b_t]$, so $[vb_{t-1}b_t]$,
then, by \eqref{eq.insertion_left_01} and Fact \ref{fact.collinear} (d),
\begin{equation}\label{eq.insertion_left_02}
\mbox{$(v) \circ \pi' \circ (b_t)$ is a collinear sequence.}
\end{equation}
By Lemma \ref{lem.insertion_order} (a) and (b),
$a_t$ is not the first point, nor among the last two points in $\pi$,
so either $a_t$ appears on $\pi'$, or there are consecutive points 
$x$ and $y$ on $\pi'$ such that $[xa_ty]$, then, by
Fact \ref{fact.collinear} (d), it can be inserted into the sequence in
\eqref{eq.insertion_left_02} to form a collinear sequence.
In either case, we have $\{v, a_1, b_1, \dots, a_t, b_t\}$ is a collinear set.
We choose a collinear sequence $\pi^*$ for the set where $a_1$ comes before
$a_2$ from Fact \ref{fact.two_ordering}.
Again by Fact \ref{fact.two_ordering}, $\pi$ must be a subsequence of $\pi^*$,
and since $[va_1b_1]$ in \eqref{eq.insertion_left_00}, so $\pi^* = (v) \circ \pi$.

\noindent {\em Case 2.} $\tau$ is the empty sequence.
This is symmetric to Case 1. In fact we can reverse $\pi$ and reduce this to Case 1.

\noindent {\em Case 3.} Both $\sigma$ and $\tau$ are non-empty.
Let $x$ be the last point in $\sigma$ and $y$ be the first point of $\tau$.
By Fact \ref{fact.collinear} (d), it is enough to show $[xvy]$.

\noindent {\em Case 3.1.} Both $x$ and $y$ are opening points.
We may assume $x = a_i$ and $y=a_{i+1}$ for some $1 \le i < t$.
By Lemma \ref{lem.insertion_order} (b) we have
\begin{equation}\label{eq.insertion_middle_01}
[a_i a_{i+1} b_i b_{i+1}]
\end{equation}
Since $v$ is on the left side of $y = a_{i+1}$, we have
$[v a_{i+1} b_{i+1}]$,
together with \eqref{eq.insertion_middle_01},
we get $[va_{i+1} b_i b_{i+1}]$ thus 
\begin{equation}\label{eq.insertion_middle_03}
[v a_{i+1} b_{i}]
\end{equation}
The point $b_{i}$ comes after $a_{i+1}$ in $\pi$, so it is in $\tau$.
Together with the fact $a_i$ is in $\sigma$, this implies $[a_ivb_i]$. And this, together with \eqref{eq.insertion_middle_03}, implies
$[a_iva_{i+1}b_i]$, so $[a_i v a_{i+1}]$, i.e., $[xvy]$.

\noindent {\em Case 3.2.} Both $x$ and $y$ are closing points.
This is completely symmetric to the previous sub-case. Again,
we may reverse $\pi$ and reduce to the previous one.

\noindent {\em Case 3.3.} $x=a_j$ is an opening point and $y=b_i$
is a closing point. 
In $\pi$, $x$ and $y$ are consecutive;
$a_i$ occurs before $b_i$, so it is before $a_j$,
$b_j$ occurs after $a_j$, so it is after $b_i$.
From here we may conclude
\begin{equation}\label{eq.insertion_middle_04}
[a_i a_j b_i b_j] \Rightarrow
\rho(a_i, b_i)+\rho(a_j, b_j) = \rho(a_i, b_j)+\rho(a_j, b_i),
\end{equation}
\begin{equation}\label{eq.insertion_middle_05}
a_i \in \sigma, b_i \in \tau \Rightarrow 
[a_ivb_i] \Rightarrow \rho(a_i, b_i) = \rho(a_i, v) + \rho(v, b_i),
\end{equation}
and 
\begin{equation}\label{eq.insertion_middle_06}
a_j \in \sigma, b_j \in \tau \Rightarrow 
[a_jvb_j] \Rightarrow \rho(a_j, b_j) = \rho(a_j, v) + \rho(v, b_j).
\end{equation}
Add the conclusions in \eqref{eq.insertion_middle_05} and \eqref{eq.insertion_middle_06},
then compare with the conclusion in \eqref{eq.insertion_middle_04}
we have
\begin{equation}\label{eq.insertion_middle_07}
(\rho(a_i, v) + \rho(v, b_j))+(\rho(a_j, v) + \rho(v, b_i)) = \rho(a_i, b_j) + \rho(a_j, b_i).
\end{equation}
But in every metric space, 
\begin{equation}\label{eq.insertion_middle_08}
\rho(a_i, v) + \rho(v, b_j) \ge \rho(a_i, b_j), \;\;
\rho(a_j, v) + \rho(v, b_i) \ge \rho(a_j, b_i).
\end{equation}
\eqref{eq.insertion_middle_07} means both inequalities in
\eqref{eq.insertion_middle_08} must take the equal sign.
In particular, $[a_j v b_i]$, i.e., $[xvy]$.

\noindent {\em Case 3.4.} $x=b_i$ is a closing point and $y=a_j$
is an opening point.
In $\pi$, $x$ and $y$ are consecutive, it is easy to conclude
\begin{equation}\label{eq.insertion_middle_09}
(a_i, a_{i+1}, b_i, a_j, b_{i+1}, b_j) \;\; \mbox{is a sub-sequence of $\pi$}.
\end{equation}
$a_i, a_{i+1}, b_i$ are in $\sigma$ and $a_j, b_{i+1}, b_j$ are in $\tau$,
so we have
\begin{equation}\label{eq.insertion_middle_10}
[a_i b_i v], [a_{i+1} v b_{i+1}], \; \mbox{and} \; [va_j b_j].
\end{equation}
$[a_i a_{i+1} b_i]$ reads from \eqref{eq.insertion_middle_09},
and with the first item in \eqref{eq.insertion_middle_10},
we get $[a_ia_{i+1}b_iv]$, so $[a_{i+1}b_iv]$,
together with the second item in \eqref{eq.insertion_middle_10}
we get
\begin{equation}\label{eq.insertion_middle_11}
[a_{i+1}b_ivb_{i+1}].
\end{equation}
$[a_jb_{i+1}b_j]$ reads from \eqref{eq.insertion_middle_09},
and with the last item in \eqref{eq.insertion_middle_10},
we get $[va_jb_{i+1}b_j]$, so $[va_jb_{i+1}]$.
$[va_jb_{i+1}]$ and \eqref{eq.insertion_middle_11}
implies $[a_{i+1}b_iva_jb_{i+1}]$,
so $[xvy]$.
\end{proof}

\begin{defi}
For every connected subgraph $H$ of $\mathcal{R}$,
a collinear ordering $\pi$ for the vertex set of $H$,
and every point $v \in L$ that does not appear on $\pi$,
we say $v$ is {\em outside} $H$ if $(v) \circ \pi$ or $\pi \circ (v)$ is a collinear sequence,
otherwise we say $v$ is {\em inside} $H$.
\end{defi}

Note that the term {\em outside} is consistent with the terminology introduced in Definition \ref{defi.collinear_set_triple}: 
if we view the pair $\{a, b\}$ as a one-vertex connected subgraph $H$ for the line $\ov{ab}$, 
then $v$ is outside this $H$ if and only if it is outside $\{a, b\}$ according to Definition \ref{defi.collinear_set_triple}.

\begin{lem}\label{lemma.P_line}
The vertex set of every connected subgraph of $\mathcal{R}$ has a collinear ordering.
\end{lem}

\begin{proof}
We prove this by induction on the order of the subgraph $H$. The base case when $H$ has
just one vertex is obvious. Now assume every connected subgraph of $H$ with a smaller order
has a collinear ordering.

By Lemma \ref{lem.non_cut_point}, there is a vertex $e = ab$ in $H$
such that $H-e$ is connected. By the inductive hypothesis, 
there is a collinear ordering $\pi$ for $H-e$.
By the connectedness of $H$, there is another vertex
$f = cd$ such that $ab \sim_g cd$,
we may assume $[acbd]$.
By Lemma \ref{lem.insertion},
$b$ can be inserted into $\pi$ to form a collinear sequence $\pi'$.
Since $[cbd]$, their order in $\pi'$ from left to the right is $c$, $b$, and $d$.
Let $x$ and $y$ be the left and right neighbour of $b$ in $\pi'$,
respectively.
$x$ and $y$ are consecutive in $\pi$, and both $x$ and $y$ are contained in the
range between $c$ and $d$ in $\pi$, inclusively.

By Lemma \ref{lem.insertion},
$a$ can be inserted into $\pi$ to form a collinear sequence $\pi''$.
$a$ can not be inserted into $\pi$ between $x$ and $y$ to form a collinear sequence,
otherwise $a$ is between $c$ and $d$, which contradicts the fact $[acd]$.
So $x$ and $y$ is still consecutive in $\pi''$. From $\pi'$ we see $[xby]$,
so $b$ can be inserted into $\pi''$ between $x$ and $y$ to form a collinear ordering.
\end{proof}

\begin{defi}
By Lemma \ref{lemma.P_line} and Fact \ref{fact.two_ordering},
the vertex set of every green component $\mathcal{C}$ in level $k$
has exactly two collinear orderings reverse each other.
We call one of them the {\em standard collinear ordering} for $\mathcal{C}$.
\end{defi}

\section{Proof of the main theorem}\label{sect.proof_main}

Our first lemma is a main ingredient in \cite{ACHKS},
a more rudimental form of this idea was first presented in \cite{KP},
we prove it here for completeness.

\begin{lem}\label{lem.long_path}
If $t$ distinct points $v_1, v_2, \dots, v_t$ satisfy
$[v_1v_2\dots v_t]$ in a metric space without a universal line,
then there are at least $t$ distinct lines.
\end{lem}

\begin{proof}
For every $i$ such that $1 \le i < t$, pick a point $u_i \in V \setminus \ov{v_iv_{i+1}}$. (It is possible that $u_i = u_j$ for $i \neq j$.)
We claim that $L_i = \ov{u_iv_i}$ ($i=1, 2, \dots, t-1$) are all distinct lines. 
First of all, $u_i \not\in \ov{v_iv_{i+1}}$, so $\{u_i, v_i, v_{i+1}\}$ is not collinear,
so $v_{i+1} \not\in L_i$ for each $i$. For $1 \le i < j < t$,
suppose $L_i = L_j$, so $\{u_i, v_i, v_j\}$ is collinear.
By Fact \ref{fact.collinear}, $[u_iv_iv_j]$ and $[v_iv_{i+1}v_j]$ imply $[u_iv_iv_{i+1}]$;
$[v_iv_ju_i]$ and $[v_iv_{i+1}v_j]$ also imply $[u_iv_iv_{i+1}]$;
finally, $[v_iu_iv_j]$ and $[v_iv_jv_{j+1}]$ imply $[v_iu_iv_{j+1}]$,
so $v_{j+1} \in L_i = L_j$, this means $\{v_{j+1}, v_j, u_j\}$ is collinear,
contradicts the fact that $u_j \not\in \ov{v_jv_{j+1}}$.
\end{proof}

\begin{defi}
Let $L$ be a line in a metric space $(V, \rho)$ without a universal line,
let $k$ be index with $2 \le k \le h(L)-1$,  
let $\mathcal{C}$ be a green component in level $k$ of $\mathcal{P}_L$,
and let $(a_1, a_2, \dots, a_t)$ be the opening sequence of $\mathcal{C}$'s standard collinear ordering.
For every $i$ such that $1 \le i < t$, $\ov{a_ia_{i+1}} \neq V$, so we can pick (any) $u_i \in V$
such that $\{u_i, a_i, a_{i+1}\}$ is not collinear. Set $L_i = \ov{u_i a_{i+1}}$
so $a_i \not\in L_i$.

We call $u_i$ the $i$-th special point and $L_i$ the $i$-th special line
with respect to the green component $\mathcal{C}$.
\end{defi}

Note that the $u_i$'s may well be repeated for different index $i$.
However, now we are going to show that the $L_i$'s are distinct.

\begin{lem}\label{lem.distinct_lines_samecpt}
For a line $L$ in a metric space $(V, \rho)$ without a universal line,
a level index $k$ with $2 \le k \le h(L)-1$,
and any green component $\mathcal{C}$ in level $k$ of $\mathcal{P}_L$,
all the special lines with respect to $\mathcal{C}$ are distinct. 
\end{lem}

\begin{proof}
Denote $(a_1, a_2, \dots, a_t)$ the opening sequence of 
the standard collinear ordering of $\mathcal{C}$;
denote $u_i$ the $i$-th special element and $L_i$ the $i$-th special line $\ov{u_i a_{i+1}}$.
Suppose the contrary that there are $i<j$ such that $L_i = L_j$. So $a_{j+1} \in L_j = L_i$,
this means $\{u_i, a_{i+1}, a_{j+1}\}$ is collinear.

\noindent {\em Case 1.} $[a_{i+1}u_i a_{j+1}]$. In $\pi$ we have $[a_ia_{i+1}a_{j+1}]$,
so $[a_ia_{i+1}u_i a_{j+1}]$, which implies $\{a_i, a_{i+1}, u_i\}$, 
contradicts the definition of $u_i$.

\noindent {\em Case 2.} $[u_i a_{i+1} a_{j+1}]$ (respectively $[a_{i+1} a_{j+1} u_i]$). 
In $\pi$ we have $[a_{i+1} a_j a_{j+1}]$ 
(no matter $i+1 = j$ or not), so $[u_i  a_{i+1} a_j a_{j+1} ] $ (respectively $[a_{i+1} a_j a_{j+1} u_i]$),
which implies $\{a_j, u_i, a_{i+1}\}$ is collinear, 
and $a_j \in L_i = L_j$, contradicts the definition of $L_j$.
\end{proof}

\begin{lem}\label{lem.distinct_lines_diffcpt}
For a line $L$ in a metric space $(V, \rho)$ without a universal line,
a level index $k$ with $2 \le k \le h(L)-1$, any two special lines
with respect to two different components in level $k$ of $\mathcal{P}_L$ are distinct. 
\end{lem}

\begin{proof}
Let $\mathcal{C}$ and $\mathcal{C}'$ be two different components. Let $(a_1, a_2, \dots, a_t)$ be the opening sequence of the standard collinear ordering
of $\mathcal{C}$, 
$b_i$ be the corresponding closing point to $a_i$;
let $u_i$ be its $i$-th special element and 
$L_{i} = u_ia_{i+1}$ be the $i$-th special line.
Let $\pi'$ be a standard collinear ordering of $\mathcal{C}'$,
with opening sequence $(a'_1, a'_2, \dots, a'_s)$,
$b'_i$ be the corresponding closing point to $a'_i$.
By Lemma \ref{lem.insertion_order} (b),
$[a_i a_{i+1} b_i]$. Since $a_ib_i$ and $a'_jb'_j$ are in different
green components, we have $a_ib_i \sim_b a'_jb'_j$,
together with $[a_i a_{i+1} b_i]$, we get 
\begin{equation}\label{eq.diffcpt_01}
\mbox{$a_{i+1}$ is outside $\{a'_j b'_j\}$ for every $j$ such that $1\le j \le s$.}
\end{equation}
Especially, $a_{i+1}$ does not appear on $\pi'$.
By Lemma \ref{lem.insertion}, $a_{i+1}$ can be inserted into
$\pi'$  to get a collinear sequence $\pi^*$.
If $a_{i+1}$ is not the first point on $\pi^*$,
let $j$ be the largest index such that $a'_j$ is earlier than $a_{i+1}$ in $\pi^*$.
If $j<s$, by the maximality of $j$, $a'_{j+1}$ is after $a_{i+1}$, and
by Lemma \ref{lem.insertion_order} (b), $b'_{j}$ is after $a'_{j+1}$
so $a_{i+1}$ is between $a'_j$ and $b'_j$ in $\pi^*$, which contradicts  \eqref{eq.diffcpt_01}.
So $j=s$, and by \eqref{eq.diffcpt_01}, $b'_s$ is also earlier than $a_{i+1}$.

In a word, $a_{i+1}$ is either the first point or the last point in $\pi^*$.
So $a_{i+1}$ is outside $\mathcal{C}'$. The lemma follows from the next stronger lemma: $a_{i+1}$ is a point of $L$ that is outside $\mathcal{C}'$, so no special line with respect to the second component contains $a_{i+1}$.
\end{proof}

\begin{lem}\label{lem.distinct_lines_exclusive}
For a line $L$ in a metric space $(V, \rho)$ without a universal line,
a level index $k$ with $2 \le k \le h(L)-1$,
and a green component $\mathcal{C}$ in level $k$,
a special line with respect to $\mathcal{C}$ does not contain any point 
$v \in L$ but outside $\mathcal{C}$. 
\end{lem}

\begin{proof}
Denote $\pi$ the standard collinear ordering of $\mathcal{C}$,
$(a_1, a_2, \dots, a_t)$ its opening sequence,
and $b_i$ the closing point corresponding to $a_i$;
denote $u_i$ the $i$-th special element and $L_i = \ov{u_i a_{i+1}}$
the $i$-th special line.

Suppose that $v\in L$ is outside $\mathcal{C}$ and assume for contradiction that there is a special line $L_i$ such that $v\in L_i$. Either $(v) \circ \pi$ or $\pi \circ (v)$ is a collinear sequence.
This means both $a_i$ and $b_{i+1}$ are collinear with $\{v, a_{i+1}\}$,
and exactly one of them is between $v$ and $a_{i+1}$, the other is outside.

Since $v \in L_i$, the set $\{v, a_{i+1}, u_i\}$ is collinear.
Depending on whether $u_i$ is between $v$ and  $a_{i+1}$ or not,
by Fact \ref{fact.collinear} (c),
there is one point in $x \in \{a_i, b_{i+1}\}$ such that 
$\{x, u_i, v, a_{i+1}\}$ is collinear, so $\{x, u_i, a_{i+1}\}$ is collinear.

If $x = a_i$, $\{u_i, a_i, a_{i+1}\}$ is collinear, which contradicts the definition of $u_i$.
If $x = b_{i+1}$, $u_i \in \ov{a_{i+1} b_{i+1}} = L$;
by Lemma \ref{lem.insertion}, $u_i$ either appears on $\pi$ 
or can be inserted into $\pi$ to form a collinear sequence,
so $\{u_i, a_i, a_{i+1}\}$ is collinear, which gain contradicts the definition of $u_i$.
\end{proof}

\begin{lem}\label{lem.distinct_lines}
For a line $L$ a metric space $(V, \rho)$ without a universal line,
a level index $k$ with $2 \le k \le h(L)-1$, all the special lines in level $k$
are distinct. 
\end{lem}

\begin{proof}
This immediately follows from Lemmas \ref{lem.distinct_lines_samecpt} and \ref{lem.distinct_lines_diffcpt}.
\end{proof}

Now we prove the main theorem.

\begin{proof} (of Theorem \ref{thm.main})
Denote $n = |V|$ and $m$ the number of distinct lines.
The generating pairs defined in \eqref{eq.def_generator_pairs} give a partition of $\binom{V}{2}$
into $m$ $K(L)$'s. If none of the parts $K(L)$ has size more than $4n^{4/3}$, we have
the number of parts $m \geq \binom{n}{2} / (4n^{4/3}) \in \Omega(n^{2/3})$.

Otherwise, there is a line $L$ with $|K(L)| \ge 4n^{4/3}$.
Consider the poset $\mathcal{P}_L$.

If there is a $k$ with $1 \le k \le h(L)$ such that the antichain $\mathcal{P}_L^{(k)}$ has more than $n^{4/3}$
elements, by the handshake lemma, we have a point $v \in V$ such that $v$ is incident to more than
$2n^{1/3}$ pairs in $\mathcal{P}_L^{(k)}$, i.e., there are points $u_1, u_2, \dots, u_d$
in $V$, such that $vu_i \in \mathcal{P}_L^{(k)}$ for all $1 \le i \le d$, and $d > 2n^{1/3}$.
For a pair of indices $i$ and $j$ where $1 \le i < j \le d$, $\ov{vu_i} = \ov{vu_j} = L$ implies $\{v, u_i, u_j\}$ is collinear,
now the fact that $\mathcal{P}_L^{(k)}$ is an antichain implies $[u_ivu_j]$.
Fact \ref{fact.star} implies that, for any $k \not\in \{i, j\}$, $\{u_i, u_j, u_k\}$
is not collinear. So $\ov{u_iu_j}$ contains both $u_i$ and $u_j$ but no any other point
in $\{u_1, u_2, \dots, u_d\}$. Therefore we have $\binom{d}{2} \in \Omega(n^{2/3})$ distinct lines.

Now we may assume for each $1 \le k \le h(L)$, we have $|\mathcal{P}_L{(k)}| \le n^{4/3}$,
this is in particular true for $k = 1$ and $k = h(L)$. Since $|K(L)| \ge 4n^{4/3}$,
we have
\begin{equation}\label{eq.normal_level_sum}
\sum_{k=2}^{h(L)-1} \left|\mathcal{P}_L{(k)}\right| = 
\sum_{k=2}^{h(L)-1} \left( \left|Q_L^{(k)} \right| + \sum_{i=1}^{c_L(k)} \left|P_L^{(k, i)}\right| \right)
\ge 2n^{4/3}.
\end{equation}
When $h(L) > n^{2/3}$, Fact \ref{fact.level_k_inner} and Lemma \ref{lem.long_path}
already imply $m \ge n^{2/3}$. Now we may assume $h(L) \le n^{2/3}$.

By Fact \ref{fact.Q_size},
\[
\sum_{k=2}^{h(L)-1} \left|Q_L^{(k)} \right| \le n \left( \frac{1}{1} + \frac{1}{2} + \dots + \frac{1}{\lfloor n^{2/3} \rfloor - 2} \right) \le n\ln n.
\]
Subtracting from \eqref{eq.normal_level_sum}, we have, for large enough $n$,
\[
\sum_{k=2}^{h(L)-1} \left( \sum_{i=1}^{c_L(k)} \left|P_L^{(k, i)}\right|\right)
\ge 2n^{4/3} - n \ln n \ge n^{4/3}.
\]
So there exists a $k$, $2 \le k \le h(L) \le n^{2/3}$, such that
$\sum_{i=1}^{c_L(k)} \left|P_L^{(k, i)}\right| \ge n^{2/3}$.
By definition, each $P_L^{(k, i)}$ has $|P_L^{(k, i)}| \ge 2$
and has exactly $|P_L^{(k, i)}| - 1$ special lines.
By Lemma \ref{lem.distinct_lines}, we have at least
\[
\sum_{i=1}^{c_L(k)} \left( \left|P_L^{(k, i)}\right|-1 \right) 
\ge \sum_{i=1}^{c_L(k)} \left(\left|P_L^{(k, i)}\right| / 2 \right) 
 \ge \frac{n^{2/3}}{2}
\] 
distinct lines in the metric space.
\end{proof}

\section{Discussion -- the structure of two green components}

The following structure of two green components can be proved by an induction.
After simplification we do not need it in the proof of the main theorem, and
we omit the proof.

\begin{fact}\label{fact.P_line_plus}
For every index $k$ such that $2 \le  k \le h(L)-1$, let $\pi$ and $\sigma$
be the standard ordering of two different green components in 
level $k$,
there is a collinear sequence $\pi_1 \circ \sigma_1$,
where $\pi_1$ is either $\pi$ or $\pi^R$,
and $\sigma_1$ is either $\sigma$ or $\sigma^R$.
\end{fact}

\section*{Acknowledgement}

The author heard this problem from the sessions in the Semicircle math club.
We would like to thank Xiaomin Chen, Qiao Sun, Zhenyuan Sun, Chunji Wang and Qicheng Xu
for the happy discussions on this work in the club.
We thank Xiaomin Chen and Va\v{s}ek Chv\'atal for their
great help in writing this paper.
We also thank two anonymous referees for their nice and 
helpful comments and suggestions.

\end{document}